\documentclass{amsart}[12pt]
\usepackage{hyperref}
\usepackage{graphicx,graphics, amsmath, amsthm, amscd, amsfonts}

\usepackage{latexsym}
\makeatletter \oddsidemargin.9375in \evensidemargin \oddsidemargin
\newtheorem{theorem}{Theorem}[section]
\newtheorem{lemma}[theorem]{Lemma}

\newtheorem{corollary}[theorem]{Corollary}
\theoremstyle{definition}

\newtheorem{example}[theorem]{Example}

\numberwithin{equation}{section}
\DeclareMathOperator{\Tr}{tr}
\begin{document}

\title[Some results on singular value inequalities]{Some results on singular value inequalities of compact operators in Hilbert space}

\author[A. Taghavi, V. Darvish, H. M. Nazari, S. S. Dragomir ]{A. Taghavi, V. Darvish, H. M. Nazari, S. S. Dragomir }

\address{{ Department of Mathematics, Faculty of Mathematical Sciences,
University of Mazandaran, P. O. Box 47416-1468, Babolsar, Iran.}}
\address{{Mathematics, School of Engineering and Science Victoria University, PO Box 14428, Melbourne City, MC 8001, Australia.}}
\address{{School of Computational and Applied Mathematics, University of the Witwatersrand, Private Bag 3, Johannesburg 2050, South Africa.}}
\email{taghavi@umz.ac.ir, vahid.darvish@vu.edu.au,
m.nazari@stu.umz.ac.ir, sever.dragomir@vu.edu.au}


 \Large
\begin{abstract}
We prove several singular value inequalities for sum and product of compact operators in Hilbert space. Some of our results generalize the previous inequalities for operators. Also, applications of some inequalities are given.

\end{abstract}
 \maketitle
\section{\textbf{Introduction}}
Let $B(H)$ stand for the $C^{*}$-algebra of  all bounded linear
operators on a complex separable Hilbert space $H$ with inner
product $\langle \cdot,\cdot\rangle$  and let $K(H)$ denote the
two-sided ideal of compact operators in $B(H)$. For $A\in B(H)$, let
$\|A\|=\sup\{\|Ax\|:\|x\|=1\}$ denote the usual operator norm of $A$
and $|A|=(A^{*}A)^{1/2}$ be the absolute value of $A$.\\
An operator $A\in B(H)$ is positive and write $A\geq0$ if $\langle
Ax,x\rangle\geq0$ for all $x\in H$. We say $A\leq B$ whenever
$B-A\geq0$.\\
 We consider the wide class of unitarily invariant
norms $|||\cdot|||$. Each of these norms is defined on an ideal in
$B(H)$ and it will be implicitly understood that when we talk of
$|||T|||$, then the operator $T$ belongs to the norm ideal
associated with $|||\cdot|||$. Each unitarily invariant norm
$|||\cdot|||$ is characterized by the invariance property
$|||UTV|||=|||T|||$ for all operators $T$ in the norm ideal
associated with $|||\cdot|||$ and for all unitary operators $U$ and
$V$ in $B(H)$.
 For $1\leq
p<\infty$, the Schatten $p$-norm of a compact operator $A$ is
defined by $\|A\|_{p}=(\Tr |A|^{p})^{1/p}$, where $\Tr$ is the usual
trace functional. Note that for $A\in K(H)$ we have,
$\|A\|=s_{1}(A)$, and if $A$ is a Hilbert-Schmidt operator, then
$\|A\|_{2}=(\sum_{j=1}^{\infty}s_{j}^{2}(A))^{1/2}$. These norms are
special examples of the more general class of the Schatten
$p$-norms, which are unitarily invariant \cite{bhabook}.

 The direct sum $A\oplus B$ denotes the block diagonal
matrix $\left[\begin{array}{cc}
A&0\\
0&B\\
\end{array}\right]$ defined on $H\oplus H$, see \cite{aud,zhan}. It
is easy to see that
\begin{equation}\label{eq12}
\|A\oplus B\|=\max(\|A\|,\|B\|),
\end{equation}
and
\begin{equation}\label{eq13}
\|A\oplus B\|_{p}=(\|A\|_{p}^{p}+\|B\|_{p}^{p})^{1/p}.
\end{equation}
We denote the singular values of an operator $A\in K(H)$ as
$s_{1}(A)\geq s_{2}(A)\geq \ldots$ are the eigenvalues of the
positive operator $|A|=(A^{*}A)^{1/2}$ and eigenvalues of the self-adjoint operator $A$ denote as $\lambda_{1}\geq\lambda_{2}\geq \ldots$ which repeated accordingly to
multiplicity.

There is a one-to-one correspondence between symmetric gauge
functions defined on sequences of real numbers and unitarily
invariant norms defined on norm ideals of operators. More precisely,
if $|||\cdot|||$ is unitarily invariant norm, then there exists a
unique symmetric gauge function $\Phi$ such that
$$|||A|||=\Phi(s_{1}(A),s_{2}(A),\ldots),$$
for every operator $A\in K(H)$. Let $A\in K(H)$, and if $U,V\in
B(H)$ are unitarily operators, then
$$s_{j}(UAV)=s_{j}(A),$$
for $j=1,2,\ldots$ and so unitarily invariant norms satisfies the
invariance property
$$|||UAV|||=|||A|||.$$

In this paper, we obtain some inequalities for sum and product
of operators. Some of our results generalize the previous inequalities for operators. 
\section{\textbf{Some singular value inequalities for sum and product of operators}}
In this section we give inequalities for singular value of
operators. Also, some norm inequalities are obtained as an
application.\\
First we should remind the following inequalities. We apply inequalities (\ref{tao2}) and (\ref{kit2}) in our proofs.

The following inequality due to Tao \cite{tao}  asserts that if $A, B, C\in K(H)$
such that $\left[\begin{array}{cc}
A&B\\
B^{*} &C\\
\end{array}\right]\geq 0$, then
\begin{equation}\label{tao2}
2s_{j}(B)\leq s_{j}\left[\begin{array}{cc}
A&B\\
B^{*} &C\\
\end{array}\right],
\end{equation}
for $j=1,2,\ldots$.\\
Here, we give another proof for above inequality.\\
Let  $\left[\begin{array}{cc}
A&B\\
B^{*} &C\\
\end{array}\right]\geq 0$ then $\left[\begin{array}{cc}
A&-B\\
-B^{*} &C\\
\end{array}\right]\geq 0$ and have the same singular values (see\cite[Theorem 2.1]{aud}). So, we can write
$$\left[\begin{array}{cc}
0&2B\\
2B^{*} &0\\
\end{array}\right]\leq \left[\begin{array}{cc}
A&B\\
B^{*} &C\\
\end{array}\right],$$
and 
$$\left[\begin{array}{cc}
0&-2B\\
-2B^{*} &0\\
\end{array}\right]\leq \left[\begin{array}{cc}
A&-B\\
-B^{*} &C\\
\end{array}\right].$$ On the other hand, we know that for every self-adjoint compact operator $X$ we have $s_{j}(X)\leq  \lambda_{j}(X\oplus -X)$, for all $j=1,2,\ldots$. By using of this fact we obtain
\begin{eqnarray*}
s_{j}\left(\left[\begin{array}{cc}
0&2B\\
2B^{*} &0\\
\end{array}\right]\right)&=&\lambda_{j}\left(\left[\begin{array}{cc}
0&2B\\
2B^{*} &0\\
\end{array}\right]\oplus \left[\begin{array}{cc}
0&-2B\\
-2B^{*} &0\\
\end{array}\right]\right)\\
&\leq&\lambda_{j}\left(\left[\begin{array}{cc}
A&B\\
B^{*} &C\\
\end{array}\right]\oplus \left[\begin{array}{cc}
A&-B\\
-B^{*} &C\\
\end{array}\right]  \right)\\
&=& s_{j}\left(\left[\begin{array}{cc}
A&B\\
B^{*} &C\\
\end{array}\right]\oplus \left[\begin{array}{cc}
A&-B\\
-B^{*} &C\\
\end{array}\right]  \right).
\end{eqnarray*}
So, we obtain $$s_{j}\left(\left[\begin{array}{cc}
0&2B\\
2B^{*} &0\\
\end{array}\right]\right)\leq  s_{j}\left(\left[\begin{array}{cc}
A&B\\
B^{*} &C\\
\end{array}\right]\oplus \left[\begin{array}{cc}
A&-B\\
-B^{*} &C\\
\end{array}\right]  \right).$$
Equivalently, 
$$2s_{j}(B \oplus B^{*}) \leq  \left(\left[\begin{array}{cc}
A&B\\
B^{*} &C\\
\end{array}\right]\oplus \left[\begin{array}{cc}
A&-B\\
-B^{*} &C\\
\end{array}\right]  \right).$$
Since $s_{j}(B)=s_{j}(B^{*})$ and  $s_{j}\left(\left[\begin{array}{cc}
A&B\\
B^{*} &C\\
\end{array}\right]\right)=s_{j}\left(\left[\begin{array}{cc}
A&-B\\
-B^{*} &C\\
\end{array}\right]\right)$, we have
$$2s_{j}(B)\leq s_{j}\left(\left[\begin{array}{cc}
A&B\\
B^{*} &C\\
\end{array}\right]\right).$$

In \cite[Remark 2.2]{aud}, Audeh and Kittaneh  proved that for every $A,B,C\in K(H)$ such that
$\left[\begin{array}{cc}
A&B\\
B^{*} &C\\
\end{array}\right]\geq 0$, then
\begin{equation}\label{asho}
s_{j}\left(\left[\begin{array}{cc}
A&B\\
B^{*} &C\\
\end{array}\right]\right)\leq 2s_{j}(A\oplus C),
\end{equation}
for $j=1,2,\ldots$. Therefore, by inequality (\ref{tao2}) we have the following inequality
\begin{equation}\label{kit2}
s_{j}(B)\leq s_{j}(A\oplus C),
\end{equation}
for $j=1,2,\ldots$.
Since every unitarily invariant norm is a monotone function of the singular values of an
operator, we can write 
\begin{equation}\label{tbo}
\left|\left|\left|\left[\begin{array}{cc}
A&B\\
B^{*} &C\\
\end{array}\right]\right|\right|\right|\leq 2|||A\oplus C|||.
\end{equation}
We can obtain the reverse of  inequality (\ref{tbo}) for arbitrary operators $X,Y\in B(H)$ by pointing out the following inequality holds because of norm property
$$|||X+Y|||\leq |||X|||+|||Y|||.$$
Replace $X$ and $Y$ by $X-Y$ and $X+Y$, respectively. We have
$$2|||X|||\leq |||X-Y|||+|||X+Y|||,$$
for all $X,Y\in B(H)$. \\
Let $X=\left[\begin{array}{cc}
A&0\\
0 &C\\
\end{array}\right]$ and $Y=\left[\begin{array}{cc}
0&B\\
B^{*} &0\\
\end{array}\right]$ in above inequality. So,
\begin{eqnarray*}
2\left|\left|\left|\left[\begin{array}{cc}
A&0\\
0 &C\\
\end{array}\right]\right|\right|\right|&\leq& \left|\left|\left|\left[\begin{array}{cc}
A&B\\
B^{*} &C\\
\end{array}\right]\right|\right|\right|+\left|\left|\left|\left[\begin{array}{cc}
A&-B\\
-B^{*} &C\\
\end{array}\right]\right|\right|\right|\\
&=&2\left|\left|\left|\left[\begin{array}{cc}
A&B\\
B^{*} &C\\
\end{array}\right]\right|\right|\right|.
\end{eqnarray*}
Hence,
$$|||
A\oplus C
|||\leq \left|\left|\left|\left[\begin{array}{cc}
A&B\\
B^{*} &C\\
\end{array}\right]\right|\right|\right|,$$
for all $A,B,C\in B(H)$. $\left[\begin{array}{cc}
A&0\\
0 &C\\
\end{array}\right]$ is called a \textit{pinching} of $\left[\begin{array}{cc}
A&B\\
B^{*}&C\\
\end{array}\right]$. \\
  For operator norm we have
$$\max\{\|A\|,\|C\|\}\leq  \left\|\left[\begin{array}{cc}
A&B\\
B^{*} &C\\
\end{array}\right]\right\|.$$

Here we give a generalization of the inequality which has been
proved by Bhatia and Kittaneh in \cite{bha1}. They have shown that
if $A$ and $B$ are two $n\times n$ matrices, then
$$s_{j}(A+B)\leq
s_{j}\left((|A|+|B|)\oplus(|A^{*}|+|B^{*}|)\right),$$ for $1\leq
j\leq n$.\\
For giving a generalization of above inequality, we need the
following lemmas.

In the rest of this section, we always assume that $f$ and $g$ are
non-negative functions on $[0, \infty)$ which are continuous and
satisfying the relation $f(t)g(t)=t$ for all $t\in [0, \infty)$.
\\
\\
The following lemma is due to Kittaneh
\cite{kit}.
\begin{lemma}\label{function}
Let $A, B$, and $C$ be operators in $B(H)$ such that $A$ and $B$ are
positive and $BC=CA$. If $\left[\begin{array}{cc}
A&C^{*}\\
C&B\\
\end{array}\right]$ is positive in $B(H\oplus H)$, then $\left[\begin{array}{cc}
f(A)^{2}&C^{*}\\
C&g(B)^{2}\\
\end{array}\right]$ is
also positive.
\end{lemma}

Let $T$ be an operator in $B(H)$. We know that
$\left[\begin{array}{cc}
|T|&T^{*}\\
T&|T^{*}|\\
\end{array}\right]\geq0$, if $T$ is normal then we have $\left[\begin{array}{cc}
|T|&T^{*}\\
T&|T|\\
\end{array}\right]\geq0$, ( see \cite{bhabook2}).

\begin{lemma}\label{lem1}
Let $A$ be an operator in $B(H)$. Then we have
\begin{equation}\label{l1}
\left[\begin{array}{cc}
|A|^{2\alpha}&A^{*}\\
A&|A^{*}|^{2(1-\alpha)}\\
\end{array}\right]\geq0,
\end{equation}
where $0\leq\alpha\leq1$.
\end{lemma}
\begin{proof}
It is easy to check that $A|A|^{2}=|A^{*}|^{2}A$, then we have
$A|A|=|A^{*}|A$ for $A\in B(H)$. Now by making use of Lemma
\ref{function}, for $f(t)=t^{\alpha}$ and $g(t)=t^{1-\alpha}$,
$0\leq\alpha\leq1$, and positivity of $\left[\begin{array}{cc}
|A|&A^{*}\\
A&|A^{*}|\\
\end{array}\right]$, we obtain the result.
\end{proof}
\begin{theorem}\label{vd1}
Let $A$ and $B$ be two operators in $K(H)$. Then we have
$$s_{j}(A+B)\leq
s_{j}\left((|A|^{2\alpha}+|B|^{2\alpha})\oplus(|A^{*}|^{2(1-\alpha)}+|B^{*}|^{2(1-\alpha)})\right),$$
for $j=1,2,\ldots$ where $0\leq\alpha\leq1$.
\end{theorem}
\begin{proof}
Since sum of two positive operator is positive, Lemma \ref{lem1} implies that $$\left[\begin{array}{cc}
|A|^{2\alpha}+|B|^{2\alpha}&A^{*}+B^{*}\\
A+B&|A^{*}|^{2(1-\alpha)}+|B^{*}|^{2(1-\alpha)}\\
\end{array}\right]\geq0,$$
 By inequality
(\ref{kit2}) we have the result.
\end{proof}

\begin{corollary}
Let $A$ and $B$ be two operators in $K(H)$. Then we have
$$s_{j}(A+B)\leq
s_{j}\left((|A|+|B|)\oplus(|A^{*}|+|B^{*}|)\right),$$ for
$j=1,2,\ldots$.
\end{corollary}
\begin{proof}
Let $\alpha=\frac{1}{2}$ in Theorem \ref{vd1}.
\end{proof}
It is easy to see that if $A$ and $B$ are normal operator in $K(H)$,
then we have
$$s_{j}(A+B)\leq
s_{j}\left((|A|+|B|)\oplus(|A|+|B|)\right),$$ for $j=1,2,\ldots$.\\
\\
\\
On the other hand, for $\alpha=1$ in Theorem \ref{vd1}, we have
\begin{eqnarray*}
s_{j}(A+B)&\leq& s_{j}(|A|^{2}+|B|^{2}\oplus 2I)\\
&=&s_{j}(|A|^{2}+|B|^{2})\cup s_{j}(2I)\\
&=&s_{j}(A^{*}A+B^{*}B)\cup s_{j}(2I),
\end{eqnarray*}
for $j=1,2,\ldots$.

\begin{theorem}\label{haj}
Let $A$,$B$ and $X$ be operators in $B(H)$ such that $X$ is compact.
Then we have the following
$$s_{j}\left(AXB^{*}\right)\leq
s_{j}\left(A^{*}f(|X|)^{2}A\oplus B^{*}g(|X^{*}|)^{2}B\right),$$ for
$j=1,2,\ldots$.
\end{theorem}
\begin{proof}
Since $\left[\begin{array}{cc}
|X|&X^{*}\\
X&|X^{*}|\\
\end{array}\right]\geq0$, by Lemma \ref{function} we have\\ $$Y=\left[\begin{array}{cc}
f(|X|)^{2}&X^{*}\\
X&g(|X^{*}|)^{2}\\
\end{array}\right]\geq0.$$ Let $Z=\left[\begin{array}{cc}
A&0\\
0&B\\
\end{array}\right]$. Since $Y$ is positive, we have
$$Z^{*}YZ=\left[\begin{array}{cc}
A^{*}f(|X|)^{2}A&A^{*}X^{*}B\\
B^{*}XA&B^{*}g(|X^{*}|)^{2}B\\
\end{array}\right]\geq0.$$ Hence,
by  inequality (\ref{kit2}), we have the desired result.
\end{proof}
\noindent In above theorem, let $X$ be a normal operator. Then we have
$$s_{j}\left(AXB^{*}\right)\leq
s_{j}\left(A^{*}f(|X|)^{2}A\oplus B^{*}g(|X|)^{2}B\right),$$ for
$j=1,2,\ldots$.

\begin{corollary}\label{131}
Let $A$, $B$ and $X$ be  operators in $B(H)$ such that $X$ is
compact. Then we have
$$s_{j}\left(AXB^{*}\right)\leq
s_{j}\left(A^{*}|X|A\oplus B^{*}|X^{*}|B\right),$$ for $j=1,2,\ldots$.
\end{corollary}
\begin{proof}
Let $f(t)=t^{\frac{1}{2}}$ and $g(t)=t^{\frac{1}{2}}$ in Theorem
\ref{haj}.
\end{proof}

Here, we apply above corollary to show that singular values of $AXB^{*}$ are dominated by  singular values of $\|X\|(A\oplus B)$. For our proof we need the following lemma.
\begin{lemma}\label{hz}\cite[p. 75]{bhabook}
Let $A,B\in B(H)$ such that $B$ is compact.  Then
$$s_{j}(AB)\leq \|A\|s_{j}(B),$$
for $j=1,2,\ldots$.
\end{lemma}
\begin{theorem}\label{ttyo}
Let $A, B,X\in B(H)$ such that $A$ and $ B$ are arbitrary compact. Then, we have
$$s_{j}(AXB^{*})\leq \|X\|s_{j}^{2}(A\oplus B),$$
for $j=1,2,\ldots$.
\end{theorem}
\begin{proof}
From Corollary \ref{131} we have
\begin{eqnarray*}
s_{j}(AXB^{*})&\leq& s_{j}(A^{*}|X|A\oplus B^{*}|X^{*}|B)\\
&=& s_{j}\left(\left[\begin{array}{cc}
A&0\\
0&B\\
\end{array}\right]^{*}\left[\begin{array}{cc}
|X|&0\\
0&|X^{*}|\\
\end{array}\right]\left[\begin{array}{cc}
A&0\\
0&B\\
\end{array}\right]\right)\\
&=& s_{j}\left(\left[\begin{array}{cc}
A&0\\
0&B\\
\end{array}\right]^{*}\left[\begin{array}{cc}
|X|^{\frac{1}{2}}&0\\
0&|X^{*}|^{\frac{1}{2}}\\
\end{array}\right]^{*}\left[\begin{array}{cc}
|X|^{\frac{1}{2}}&0\\
0&|X^{*}|^{\frac{1}{2}}\\
\end{array}\right]\left[\begin{array}{cc}
A&0\\
0&B\\
\end{array}\right]\right)\\
&=&s_{j}\left(\left(\left[\begin{array}{cc}
|X|^{\frac{1}{2}}&0\\
0&|X^{*}|^{\frac{1}{2}}\\
\end{array}\right]\left[\begin{array}{cc}
A&0\\
0&B\\
\end{array}\right]\right)^{*}\left(\left[\begin{array}{cc}
|X|^{\frac{1}{2}}&0\\
0&|X^{*}|^{\frac{1}{2}}\\
\end{array}\right]\left[\begin{array}{cc}
A&0\\
0&B\\
\end{array}\right]\right)\right)\\
&=&s_{j}\left(\left|\left[\begin{array}{cc}
|X|^{\frac{1}{2}}&0\\
0&|X^{*}|^{\frac{1}{2}}\\
\end{array}\right]\left[\begin{array}{cc}
A&0\\
0&B\\
\end{array}\right]\right|^{2}\right)\\
&=& s_{j}^{2}\left(\left[\begin{array}{cc}
|X|^{\frac{1}{2}}&0\\
0&|X^{*}|^{\frac{1}{2}}\\
\end{array}\right]\left[\begin{array}{cc}
A&0\\
0&B\\
\end{array}\right]\right)\\
&\leq& \left\|\left[\begin{array}{cc}
|X|^{\frac{1}{2}}&0\\
0&|X^{*}|^{\frac{1}{2}}\\
\end{array}\right]\right\|^{2}s_{j}^{2}(A\oplus B)\\
&=&\|X\|s_{j}^{2}(A\oplus B),
\end{eqnarray*}
for $j=1,2,\ldots$. The last inequality follows by Lemma \ref{hz}.
\end{proof}
In Theorem \ref{ttyo}, let $A$ and $B$ be positive operators in $K(H)$. Then
we have
\begin{equation}\label{maj}
s_{j}(A^{\frac{1}{2}}XB^{\frac{1}{2}})\leq \|X\|s_{j}(A\oplus B),
\end{equation}
for $j=1,2,\ldots$.

\begin{corollary}\label{gb}
Let $A$ and $B$ be two operators in $K(H)$. Then we have
\begin{equation}\label{cor1}
s_{j}(AB^{*})\leq s_{j}\left(A^{*}A\oplus B^{*}B\right),
\end{equation}
 for
$j=1,2,\ldots$.
\end{corollary}
\begin{proof}
Let $X=I$  in
Corollary \ref{131}.
\end{proof}
Moreover, we can write inequality (\ref{cor1}) in the following form
\begin{eqnarray*}
s_{j}(AB^{*})&\leq& s_{j}(|A|^{2}\oplus |B|^{2})\\
&=&s_{j}^{2}(|A|\oplus|B|)=s_{j}^{2}(A\oplus B),
\end{eqnarray*}
for $j=1,2,\ldots$.\\
We should note here that inequality (\ref{cor1}) can be obtained by
Theorem 1 in \cite{bha3} and Corollary 2.2 in \cite{hir}.\\
Here, we give two results of Corollary \ref{gb}.
As the first application, let $A=\left[\begin{array}{cc}
X&Y\\
0&0\\
\end{array}\right]$ and $B=\left[\begin{array}{cc}
Y&-X\\
0&0\\
\end{array}\right]$, such that $X,Y\in K(H)$ then by easy computations we have
$$s_{j}(XY^{*}-YX^{*})\leq
s_{j}((XX^{*}+YY^{*})\oplus(XX^{*}+YY^{*})),$$ for $j=1,2,\ldots$.\\

For obtaining second application, replace $A$ and $B$ in (\ref{cor1}) by $AX^{\alpha}$ and $BX^{(1-\alpha)}$ respectively, where $X$ is a compact positive operator and $\alpha\in \mathbb{R}$. So, we have
\begin{eqnarray*}
s_{j}(AXB^{*})&\leq& s_{j}\left(X^{\alpha}A^{*}AX^{\alpha}\oplus X^{(1-\alpha)}B^{*}BX^{(1-\alpha)}\right)\\
&=& s_{j}\left(\left[\begin{array}{cc}
X^{\alpha}A^{*}AX^{\alpha}&0\\
0&X^{(1-\alpha)}B^{*}BX^{(1-\alpha)}\\
\end{array}\right]    \right)\\
&=&s_{j}\left(\left[\begin{array}{cc}
X^{\alpha}A^{*}&0\\
0&X^{(1-\alpha)}B^{*}\\
\end{array}\right] \left[\begin{array}{cc}
AX^{\alpha}&0\\
0&BX^{(1-\alpha)}\\
\end{array}\right]    \right)\\
&=&s_{j}\left( \left[\begin{array}{cc}
AX^{\alpha}&0\\
0&BX^{(1-\alpha)}\\
\end{array}\right]\left[\begin{array}{cc}
X^{\alpha}A^{*}&0\\
0&X^{(1-\alpha)}B^{*}\\
\end{array}\right]    \right)\\
&=&s_{j}\left( \left[\begin{array}{cc}
AX^{2\alpha}A^{*}&0\\
0&BX^{2(1-\alpha)}B^{*}\\
\end{array}\right]\right)\\
&=& s_{j}(AX^{2\alpha}A^{*}\oplus BX^{2(1-\alpha)}B^{*}),
\end{eqnarray*}
for all $j=1,2,\ldots$.\\
Finally, we have 
\begin{equation}\label{embr}
s_{j}(AXB^{*})\leq s_{j}(AX^{2\alpha}A^{*}\oplus BX^{2(1-\alpha)}B^{*}),
\end{equation}
for all $j=1,2,\ldots$.\\
By a similar proof of Theorem \ref{ttyo} to inequality (\ref{embr}), we obtain
$$s_{j}(AXB^{*})\leq \max\{\|X^{2\alpha}\|,\|X^{2(1-\alpha)}\|\}s_{j}^{2}(A\oplus B),$$
for all $j=1,2,\ldots$.\\
In above inequality, for positive operators $A$ and $B$ in $K(H)$ we have
$$ s_{j}(A^{\frac{1}{2}}XB^{\frac{1}{2}})\leq \max\{\|X^{2\alpha}\|,\|X^{2(1-\alpha)}\|\}s_{j}(A\oplus B)$$
for all $j=1,2,\ldots$.

\section{\textbf{Some singular value inequalities for normal operators}}\label{section4}
Here we give some results for compact normal operators. For
every operator $A$, the Cartesian decomposition is to write
$A=\Re(A)+i\Im(A)$, where $\Re(A)=\frac{A+A^{*}}{2}$ and
$\Im(A)=\frac{A-A^{*}}{2i}$. If $A$ is normal operator then $\Re(A)$
and $\Im(A)$ commute together and vice versa.
\begin{theorem}\label{normal}
Let $A_{1},A_{2},\ldots,A_{n}$ be normal operators in $K(H)$. Then
we have
\begin{eqnarray*}
\frac{1}{\sqrt{2}}s_{j}\left(\oplus_{i=1}^{n}(\Re(A_{i})+\Im(A_{i}))\right)&\leq& s_{j}(\oplus_{i=1}^{n}A_{i})\\
&\leq&
s_{j}\left(\oplus_{i=1}^{n}(|\Re(A_{i})|+|\Im(A_{i})|)\right),
\end{eqnarray*}
 for
$j=1,2,\ldots$.
\end{theorem}
\begin{proof}
Let $A_{1},A_{2},\ldots,A_{n}$ be normal operators, then
$$\oplus_{i=1}^{n}A_{i}=   \left(
    \begin{array}{ccccc}
    A_{1}    &        &    &   0  &  \\
    & A_{2}      &    &    &  \\
        &  & & \ddots   &  \\
          &  0       &       &  & A_{n}\\
  \end{array}\right)$$
is normal, so we have
$$(\oplus_{i=1}^{n}\Re(A_{i}))(\oplus_{i=1}^{n}\Im(A_{i}))=(\oplus_{i=1}^{n}\Im(A_{i}))(\oplus_{i=1}^{n}\Re(A_{i}))).$$
By above equation, we obtain the following
$$\sqrt{(\oplus_{i=1}^{n}A_{i})^{*}(\oplus_{i=1}^{n}A_{i})}=\sqrt{(\oplus_{i=1}^{n}\Re(A_{i}))^{2}+(\oplus_{i=1}^{n}\Im(A_{i}))^{2}}.$$
So
\begin{eqnarray*}
s_{j}(\oplus_{i=1}^{n}A_{i})&=&s_{j}(|\oplus_{i=1}^{n}A_{i}|)\\
&=&s_{j}\left(\sqrt{(\oplus_{i=1}^{n}A_{i})^{*}(\oplus_{i=1}^{n}A_{i})}\right)\\
&=&s_{j}\left(\sqrt{(\oplus_{i=1}^{n}\Re(A_{i}))^{2}+(\oplus_{i=1}^{n}\Im(A_{i}))^{2}}\right),
\end{eqnarray*}
for $j=1,2,\ldots$.\\ By using Weyl's monotonicity principle \cite{bhabook} and the inequality\\
 $$\sqrt{(\oplus_{i=1}^{n}\Re(A_{i}))^{2}+(\oplus_{i=1}^{n}\Im(A_{i}))^{2}}\leq
|\oplus_{i=1}^{n}\Re(A_{i})|+|\oplus_{i=1}^{n}\Im(A_{i})|,$$ we have
the following
$$s_{j}\left(\sqrt{(\oplus_{i=1}^{n}\Re(A_{i}))^{2}+(\oplus_{i=1}^{n}\Im(A_{i}))^{2}}\right)\leq s_{j}(|\oplus_{i=1}^{n}\Re(A_{i})|+|\oplus_{i=1}^{n}\Im(A_{i})|),$$
for $j=1,2,\ldots$. Now for proving left side inequality, we recall
the following inequality $$0\leq
(\Re(A_{i})+\Im(A_{i})^{*}(\Re(A_{i})+\Im(A_{i})\leq
2(\Re(A_{i})^{2}+\Im(A_{2})^{2}).$$ Therefore, by using the Weyl's
monotonicity principle we can write
$$s_{j}\left(\sqrt{\left((\oplus_{i=1}^{n}\Re(A_{i}))+(\oplus_{i=1}^{n}\Im(A_{i}))\right)^{*}\left((\oplus_{i=1}^{n}\Re(A_{i}))+(\oplus_{i=1}^{n}\Im(A_{i}))\right)}\right),$$
which is less than $$
\sqrt{2}s_{j}\left(\sqrt{(\oplus_{i=1}^{n}\Re(A_{i}))^{2}+(\oplus_{i=1}^{n}\Im(A_{i}))^{2}}\right).$$
 for $j=1,2,\ldots$.
Therefore,
\begin{eqnarray*}
s_{j}((\oplus_{i=1}^{n}\Re(A_{i}))+(\oplus_{i=1}^{n}\Im(A_{i})))
&=&s_{j}(|(\oplus_{i=1}^{n}\Re(A_{i}))+(\oplus_{i=1}^{n}\Im(A_{i}))|)\\
&\leq&
\sqrt{2}s_{j}(\sqrt{(\oplus_{i=1}^{n}\Re(A_{i}))^{2}+(\oplus_{i=1}^{n}\Im(A_{i}))^{2}}).
\end{eqnarray*}
 for $j=1,2,\ldots$.
\end{proof}
The following example shows that normal condition is necessary.
\begin{example}
Let $A=\left[\begin{array}{cc}
-1+i&1\\
i&1+2i\\
\end{array}\right]$, then a calculation shows $$s_{2}(\Re(A)+i\Im(A))\approx
1.34>s_{2}(|\Re(A)|+|\Im(A)|)\approx 1.27.$$
\end{example}
\begin{corollary}\label{coro1}
Let $A$ be a normal operator in $K(H)$. Then we have
$$(1/\sqrt{2})s_{j}(\Re(A)+\Im(A))\leq s_{j}(A)\leq s_{j}(|\Re(A)|+|\Im(A)|),$$
for $j=1,2,\ldots$.
\end{corollary}
For each complex number $x=a+ib$, we know the following inequality holds
\begin{equation}\label{vatt}
\frac{1}{\sqrt{2}}|a+b|\leq |x|\leq |a|+|b|.
\end{equation}
Now, by applying Corollary \ref{coro1}, we can obtain operator version of inequality (\ref{vatt}).\\

Here, we determine
the upper and lower bound for $A+iA^{*}$.
\begin{theorem}
Let $A_{1},A_{2},\ldots,A_{n}$ be in $K(H)$. Then
\begin{eqnarray*}
\sqrt{2}s_{j}\left(\oplus_{i=1}^{n}(\Re(A_{i})+\Im(A_{i}))\right)&\leq& s_{j}(\oplus_{i=1}^{n}(A_{i}+iA_{i}^{*}))\\
&\leq&
2s_{j}(\oplus_{i=1}^{n}(\Re(A_{i})+\Im(A_{i}))),
\end{eqnarray*}
 for
$j=1,2,\ldots$.
\end{theorem}
\begin{proof}
Note that $A_{i}+iA^{*}_{i}$ is normal operator for $i=1,\ldots,n$,
so $T=\oplus_{i=1}^{n}(A_{i}+iA_{i}^{*})$ is normal. On the other
hand, we can write $T=\Re(T)+i\Im(T)$ where
$$\Re(T)=(\oplus_{i=1}^{n}(A_{i}+A_{i}^{*})+i\oplus_{i=1}^{n}(A_{i}^{*}-A_{i}))/2,$$
$$\Im(T)=(\oplus_{i=1}^{n}(A_{i}-A_{i}^{*})+i\oplus_{i=1}^{n}(A_{i}^{*}+A_{i}))/2i.$$
It is enough to compare $\Re(T)$ and $\Im(T)$ to see
$\Re(T)=\Im(T)$. So
\begin{equation}\label{mosavi}
\Re(T)+\Im(T)=\oplus_{i=1}^{n}(A_{i}+A_{i}^{*})+i\oplus_{i=1}^{n}(A_{i}^{*}-A_{i}).
\end{equation}
Now apply Theorem \ref{normal}, we have
\begin{eqnarray}
(1/\sqrt{2})s_{j}(\Re(T)+\Im(T))&\leq& s_{j}(\Re(T)+i\Im(T))\nonumber\\
&\leq&
s_{j}(|\Re(T)|+|\Im(T)|),\label{111}
\end{eqnarray}
for $j=1,2,\ldots$. Put (\ref{mosavi}),
$\Re(T)+i\Im(T)=\oplus_{i=1}^{n}(A_{i}+iA_{i}^{*})$ and $\Re(T)$ in
(\ref{111}) to obtain
\begin{equation}
(1/\sqrt{2})s_{j}(\oplus_{i=1}^{n}(A_{i}+A_{i}^{*})+i\oplus_{i=1}^{n}(A_{i}^{*}-A_{i}))\leq
s_{j}(\oplus_{i=1}^{n}(A_{i}+iA_{i}^{*})),
\end{equation}
and
\begin{eqnarray*}
s_{j}(\oplus_{i=1}^{n}(A_{i}+iA_{i}^{*})) &\leq
2s_{j}(\oplus_{i=1}^{n}(A_{i}+A_{i}^{*})/2
+i\oplus_{i=1}^{n}(A_{i}^{*}-A_{i})/2)\\
&=s_{j}(\oplus_{i=1}^{n}(A_{i}+A_{i}^{*})+i\oplus_{i=1}^{n}(A_{i}^{*}-A_{i})),
\end{eqnarray*}
for $j=1,2,\ldots$.  By writing
$\Re(\oplus_{i=1}^{n}A_{i})=\oplus_{i=1}^{n}(A_{i}+A_{i}^{*})/2$ and
$\Im(\oplus_{i=1}^{n}A_{i})=\oplus_{i=1}^{n}(A_{i}-A_{i}^{*})/2i$ we
have
\begin{eqnarray*}
(1/\sqrt{2})s_{j}(2\Re(\oplus_{i=1}^{n}A_{i})+2\Im(\oplus_{i=1}^{n}A_{i}))&\leq&
s_{j}(\oplus_{i=1}^{n}(A_{i}+iA_{i}^{*}))\\
&\leq&
s_{j}(2\Re(\oplus_{i=1}^{n}A_{i})+2\Im(\oplus_{i=1}^{n}A_{i})),
\end{eqnarray*}
 for
$j=1,2,\ldots$. Finally
\begin{eqnarray*}
\sqrt{2}s_{j}\left(\oplus_{i=1}^{n}(\Re(A_{i})+\Im(A_{i}))\right)&\leq& s_{j}(\oplus_{i=1}^{n}(A_{i}+iA_{i}^{*}))\\
&\leq&
2s_{j}(\oplus_{i=1}^{n}(\Re(A_{i})+\Im(A_{i}))),
\end{eqnarray*}
 for
$j=1,2,\ldots$.
\end{proof}

\end{document}